\documentclass[10pt, a4paper]{article}
\usepackage{color,url,latexsym,amssymb,amsthm,amsmath,amsxtra,amsfonts}
\usepackage{amsfonts}
\usepackage{amsbsy}
\usepackage{amssymb}
\usepackage{amsmath}
\usepackage{amsmath,amscd}
\usepackage{amsthm}
\usepackage{fancyhdr}

\theoremstyle{plain}
\newtheorem{theorem}{Theorem}[section]
\newtheorem{lemma}{Lemma}[section]
\newtheorem{corollary}{Corollary}[section]
\newtheorem{proposition}{Proposition}[section]
\theoremstyle{definition}
\newtheorem{definition}{Definition}[section]
\newtheorem{remark}{Remark}[section]
\newtheorem{example}{Example}[section]

\begin{document}

\title{Contact structures on Lie algebroids}
\author{Cristian Ida and Paul Popescu}
\date{}
\maketitle
\begin{abstract}
In this paper we generalize the main notions from the geometry of (almost) contact manifolds in the category of Lie algebroids. Also, using the framework of generalized geometry, we obtain an (almost) contact Riemannian Lie algebroid structure on a vertical Liouville distribution over the big-tangent manifold of a Riemannain manifold.
\end{abstract}

\medskip 
\begin{flushleft}
\strut \textbf{2010 Mathematics Subject Classification:} 22A22, 17B66, 53C15, 53C25.

\end{flushleft}

\section{Introduction}
\setcounter{equation}{0}
The importance of contact and symplectic geometry is without question. Contact manifolds can be viewed as an odd-dimensional counterpart of symplectic manifolds. Both contact and symplectic geometry are motivated by the mathematical formalism of classical mechanics, where one can consider either the even-dimensional phase space of a mechanical system or the odd-dimensional extended phase space that includes the time variable. For more about contact geometry the reader can consult the outstanding works \cite{Bl1, Bl, B-G}.

On the other hand, in the last decades, the Lie algebroids have an important place in the context of some different categories in differential geometry and mathematical physics and represent an active domain of research. The Lie algebroids, \cite{Mack2}, are generalizations of Lie algebras and integrable distributions. In fact a Lie algebroid is an anchored vector bundle with a Lie bracket on module of sections and many geometrical notions which involves the tangent bundle were generalized to the context of Lie algebroids.  In the category of almost complex geometry the notion of almost complex Lie algebroid over almost complex manifolds was introduced in \cite{B-R} as a natural extension of the notion of an almost complex manifold to that of an almost complex Lie algebroid. More generally, in \cite{A-C-N, Cr-Ida, I-Po, Po2}, is considered the notion of almost complex Lie algebroid over a smooth manifold and some problems concerning the geometry of almost complex Lie algebroids over smooth manifolds are studied in relation with corresponding notions from the geometry of almost complex manifolds. Taking into account the major role of (almost) complex geometry in the study of (almost) contact geometry, a natural generalization of (almost) contact geometry of manifolds to that of (almost) contact Lie algebroids can be of some interests. We notice that, for the particular class of Lie algebroids defined by the tangent bundle along the leaves of a foliation of odd dimension, the contact structures are introduced and studied in some recent papers \cite{Ca-Pin-Pres, Pin-Pres} under the name of foliated contact structures. In general, the notion of contact Lie algebroid appear in some very recent talks (see \cite{P-C-M0, P-C-M1, P-C-M2}), where this notion is used in order to obtain Jacobi manifolds on spheres of linear Poisson manifolds with a bundle metric. Also, the Albert cosymplectic and contact reduction theorems are extended in the Lie algebroid framework, and this reduction theory can represent a rich source in obtaining some new examples of cosymplectic or contact Lie algebroids (see \cite{P-C-M0}). The study of symplectic Lie algebroids and their reductions can be found for instance in \cite{I-M-D-M-P}. 

Our aim in this paper is to generalize some basic facts from the (almost) contact geometry on odd dimensional manifolds (see \cite{Bl1, Bl, B-G, Pi}), in the framework of Lie algebroids of odd rank and to present new examples of contact Lie algebroids. This generalization is possible mainly using the differential calculus on Lie algebroids: exterior derivative, interior product, Lie derivative (see for instance \cite{M}), but also using the connections theory on Lie algebroids (see \cite{Fe}), and the technique of Riemannian geometry on Lie algebroids (see \cite{Bo}). 

The paper is organized as follows. In the second section we present the almost contact and almost contact Riemannian structures on Lie algebroids of odd rank and we give the main properties of these structures in relation with similar properties from the case of almost contact manifolds. In the third section we present the normal almost contact structures on Lie algebroids, we study these structures,  and using the definition of the direct product of two Lie algebroids (see \cite{Mack2}), we characterize the direct product of two Lie algebroids endowed with some additional (almost Hermitian and almost contact Riemannian) structures. In the fourth section we give the basic definitions and results about contact structures on Lie algebroids in relation with similar notions from contact manifolds theory, we present some examples (see \cite{P-C-M1, P-C-M2}), we present a bijective corespondence between contact Riemannian structures and almost contact Riemannian structures on Lie algebroids, and we give some characterizations of contact Riemannian Lie algebroids. Also, the notions of $K$-contact, Sasakian and Kenmotsu Lie algebroids are introduced and some of their properties are studied as in the manifolds case. In the last section, using the framework of generalized geometry and starting from the geometry of big-tangent manifold introduced and intensively studied in \cite{Va13}, we obtain an (almost) contact Riemannian structure on the vertical Liouville distribution over the big-tangent manifold of a paracompact manifold $M$ which admits a Riemannian metric $g$. More exactly, we construct a vertical framed Riemannian $f(3,1)$-structure on the vertical bundle over the big-tangent manifold of a Riemannian manifold $(M,g)$, and when we restrict this structure to a vertical Liouville distribution which is integrable (so it is a Lie algebroid), we obtain an (almost) contact Riemannian structure on this Lie algebroid. 

Other problems and some future works can be adressed, as for instance: the study of deformations of Sasakian structures on Lie algebroids, the study of curvature problems on contact Riemannian Lie algebroids, $K$-contact, Sasaki and Kenmotsu Lie algebroids as well as the study of $F_E$-sectional curvature and Schur type theorem on Sasakian Lie algebroids. Also, taking into account the recent harmonic theory on Riemannian Lie algebroids (see \cite{Ba1}), a harmonic and $C$-harmonic theory for differential forms on Sasakian Lie algebroids can be investigated, since every almost contact Lie algebroid will be invariantly oriented (see Corollary \ref{c1}). Another important problem is that of integrability of Jacobi structures, being close related to that of Poisson structures and giving rise to contact groupoids. The progress in this direction is described at large in the recent paper \cite{C-S}, 
but we do not consider here relations with contact groupoids, which would involve some more problems, beyond the scope of our work.

The main notions introduced here are natural generalizations from the category of manifolds to the category of Lie algebroids and the most proofs are similarly with the ones given for the case of (almost) contact manifolds (see for instance \cite{Bl1, Bl, Pi}). For this reason they are omitted here.

\section{Almost contact Lie algebroids}
\setcounter{equation}{0}
In this section we define the almost contact and almost contact Riemannian structures on Lie algebroids and some properties of these structures are analyzed by analogy with the almost contact manifolds case (see \cite{Bl1, Bl, Pi}).

Let $p:E\rightarrow M$ be a vector bundle of \textit{rank} $m$ over a smooth $n$-dimensional manifold $M$, and $\Gamma(E)$ the $C^\infty(M)$-module of sections of $E$. A \textit{Lie algebroid structure} on $E$ is given by a triplet $(E,\rho_E,[\cdot,\cdot]_E)$, where $[\cdot,\cdot]_E$ is a Lie bracket on $\Gamma(E)$ and $\rho_E:E\rightarrow TM$ is called the \textit{anchor map}, such that if we also denote by $\rho_E:\Gamma(E)\rightarrow\mathcal{X}(M)$ the homomorphism of $C^\infty(M)$-modules induced by the anchor map then we have
\begin{equation}
\label{Liealg}
[s_1,fs_2]_E=f[s_1,s_2]_E+\rho_E(s_1)(f)s_2\,,\,\,\forall\,s_1,s_2\in\Gamma(E),\,\,\forall\,f\in C^\infty(M).
\end{equation}
\begin{remark}
If $(E,\rho_E,[\cdot,\cdot]_E)$ is a Lie algebroid over $M$, then the anchor map $\rho_E:\Gamma(E)\rightarrow\mathcal{X}(M)$ is a homomorphism between the Lie algebras $(\Gamma(E),[\cdot,\cdot]_E)$ and $(\mathcal{X}(M),[\cdot,\cdot])$.
\end{remark}
The exterior derivative on Lie algebroids is defined by 
\begin{equation}
(d_E\omega)(s_0,\ldots,s_p)=\sum_{i=0}^p(-1)^i\rho_E(s_i)(\omega(s_0,\ldots,\widehat{s_i},\ldots,s_p))
\label{I11}
\end{equation}
\begin{displaymath}
+\sum_{i<j=1}^p(-1)^{i+j}\omega([s_i,s_j]_E,s_0,\ldots,\widehat{s_i},\ldots,\widehat{s_j},\ldots,s_p),
\end{displaymath}
for $\omega\in \Omega^p(E)$ and $s_0,\ldots,s_p\in\Gamma(E)$, where $\Omega^p(E)$ is the set of $p$-forms on $E$. For more details about Lie algebroids and all calculus on Lie algebroids (interior product, Lie derivative etc.), we refer for instance to \cite{Fe, H-M, L-M-M, Mack2, M, L-Po}.

Let $(E,\rho_E,[\cdot,\cdot]_E)$ be a Lie algebroid of ${\rm rank}\,E=2m+1$ over a smooth $n$-dimensional manifold $M$. If there are  a section $\xi\in\Gamma(E)$, $1$--form $\eta\in\Gamma(E^*)$ and a $(1,1)$--tensor $F_E\in\Gamma(E\otimes E^*)$ such that
\begin{equation}
\label{II1}
F^2_E=-I_E+\eta\otimes \xi\,,\,\eta(\xi)=1,
\end{equation}
where $I_E$ denotes the Kronecker tensor  on $E$, then we say that $(F_E,\xi,\eta)$ is an \textit{almost contact structure} on the Lie algebroid $(E,\rho_E,[\cdot,\cdot]_E)$ or $(E,\rho_E,[\cdot,\cdot]_E, F_E,\xi,\eta)$ is an \textit{almost contact Lie algebroid}. The $1$--section $\xi$ is called \textit{Reeb section} or \textit{fundamental section}. Obviously, the set $\Gamma_\xi(E)=\{f\xi\,|\,f\in\mathcal{F}(M)\}$ has a module structure over $\mathcal{F}(M)$ and a Lie algebra structure called the \textit{Lie algebra of Reeb sections}.

Let $D_x=\{s_x\in E_x\,|\,\imath_{s_x}\eta_x=0\}\subseteq E_x$ for $x\in M$. Then $D=\cup_{x\in M}D_x$ is a vector subbundle of $E$ of rank $2m$ called \textit{contact subbundle} of $(E,\rho_E,[\cdot,\cdot]_E, F_E,\xi,\eta)$. We notice that $D=\ker\eta={\rm im}\,F_E$.

\begin{remark}
The above definition of the almost contact structure from \eqref{II1} does not depend on the anchor $\rho_E$ and the bracket $[\cdot,\cdot]_E$, hence it can be considered for a general vector bundle $E\rightarrow M$ of odd rank which will be referred to as an \textit{almost contact bundle}.
\end{remark}

Now, let us briefly present some basic properties of almost contact structures on  Lie algebroids (or general vector bundles when the notions does not depend the anchor or bracket).

\begin{proposition}
\label{pII1}
If $(F_E,\xi,\eta)$ is an almost contact structure on the vector bundle $E$ then:
\begin{enumerate}
\item[(i)] $F_E(\xi)=0$; (ii) $F_E^3=-F_E$; (iii) $\eta\circ F_E=0$; (iv) ${\rm rank}\,F_E=2m$.
\end{enumerate}
\end{proposition}
\begin{proof}
Follows in a similar manner as for almost contact manifolds (see \cite{Bl1, Bl}).
\end{proof}
Also, the following theorem holds.
\begin{theorem}
\label{tII1}
Let $E$ be a vector bundle with an almost contact structure $(F_E,\xi,\eta)$. There exists on $E$ a fiber-wise Riemannian  metric (or simply Riemannian metric) $g_E$ with the property
\begin{equation}
\label{II3}
g_E\left(F_E(s_1),F_E(s_2)\right)=g_E(s_1,s_2)-\eta(s_1)\eta(s_2)
\end{equation}
for any $s_1,s_2\in\Gamma(E)$.
\end{theorem}
\begin{proof}
We recall that a Riemannian metric in the vector bundle $p:E\rightarrow M$ is a mapping $g_E$ that assigns to every $x\in M$ a scalar product $g_E(x)$ in the local fiber $E_x$ such that for every local sections $s_1,s_2\in\Gamma(E)$, the function $x\mapsto g_E(x)(s_1,s_2)$ is smooth. Since $E$ is paracompact, there exists a Riemannian metric $g_E^{**}$ on $E$ and then, we define $g_E$ by
\begin{equation}
\label{II4}
g_E(s_1,s_2)=\frac{1}{2}\left[g_E^*(F_E(s_1),F_E(s_2))+g_E^*(s_1,s_2)+\eta(s_1)\eta(s_2)\right],
\end{equation}
where $g_E^*$ has the expression $g_E^*(s_1,s_2)=g_E^{**}\left(F_E^2(s_1),F_E^2(s_2)\right)+\eta(s_1)\eta(s_2)$. Then, it is easy to check that $g_E$ given by \eqref{II4} is a Riemannian metric on $E$ and satisfies the condition \eqref{II3}.
\end{proof}
The vector bundle $E$ with the almost contact structure $(F_E,\xi,\eta)$ and the Riemannian metric $g_E$ satisfying the condition \eqref{II3} is called an \textit{almost contact Riemannian vector bundle} or an \textit{almost contact Riemannian Lie algebroid} (when this is the case) and $(F_E,\xi,\eta,g_E)$ is an \textit{almost contact Riemannian structure} on $E$. Sometimes, we say that $g_E$ is a metric \textit{compatible} with the almost contact structure $(F_E,\xi,\eta)$. 

In a similar manner with the case of almost contact manifolds (see \cite{Bl1, Bl}), some elementary but useful properties of such metrics are specified in the following:
\begin{proposition}
\label{pII2}
If $g_E$ is a metric compatible with the almost contact structure $(F_E,\xi,\eta)$ on the vector bundle $E$ of rank $2m+1$ then:
\begin{enumerate}
\item[(i)] $\eta(s)=g_E(s,\xi)$ for all $s\in\Gamma(E)$;
\item[(ii)] on the domain $U$ of each local chart from $M$ there exists an orthonormal basis of local sections of $E$ over $U$, $\{s_1,\ldots,s_n,F_E(s_1),\ldots,F_E(s_n),\xi\}$;
\item[(iii)] $F_E+\eta\otimes \xi$ and $-F_E+\eta\otimes \xi$ are orthogonal transformations with respect to metric $g_E$;
\item[(iv)] $g_E(F_E(s_1),s_2)=-g_E(s_1,F_E(s_2))$ for every $s_1,s_2\in\Gamma(E)$.
\end{enumerate}
\end{proposition}
The local basis of sections of $E$, $\{s_1,s_2,\ldots,s_m,s_{1^*}=F_E(s_1),s_{2^*}=F_E(s_2),\ldots,s_{m^*}=F_E(s_m),\xi\}$ obtained above and denoted sometimes by $\{s_a,s_{a^*},\xi\}$, $a=1,\ldots,m$ is called a \textit{$F_E$--basis} for the almost contact Riemaniann vector bundle $(E,F_E,\xi,\eta,g_E)$ (or Lie algebroid, when this is the case). The existence of metrics compatible with an almost contact structure $(F_E,\xi,\eta)$ on $E$ allows us to state the following characterization of almost contact bundles (or Lie algebroids) by means of the structure group of the vector bundle $E$.
\begin{theorem}
\label{tII2}
The structure group of the almost contact vector bundle (or Lie algebroid) $E$ of rank $2m+1$ reduces to $U(m)\times 1$. Conversely, if the structure group of the vector bundle $E$ reduces to $U(m)\times 1$ then $E$ has an almost contact structure.
\end{theorem}
\begin{proof}
The proof follows in the same manner as for almost contact manifolds case (see for instance \cite{Bl1, Pi}). However, we briefly present here its generalization to general vector bundles case. Let $g_E$ be a metric on $E$ compatible with the almost contact structure $(F_E,\xi,\eta)$ and consider two open neighborhoods $U,V$ on $M$ trivializing $E$ with $U\cap V\neq\emptyset$. Also, we denote by $\mathcal{B}_U=\{s_a,s_{a^*},\xi\}$ and $\mathcal{B}_V=\{s^\prime_a,s^\prime_{a^*},\xi\}$ the corresponding $F_E$--bases from Proposition \ref{pII2} (ii). The matrix $(F_E)$ of $F_E$ with respect to these bases is
$$(F_E)=\left( 
\begin{array}{ll}
0\,& -I_m\,\,\,\,\,\,0 \\ 
I_m\, & \,\,\,\,\,\,0\,\,\,\,\,\,\,\,\,0\\
0\, &\,\,\,\,\,\,0\,\,\,\,\,\,\,\,\,0
\end{array}
\right).$$ 
For $x\in U\cap V$ and $s_x\in E_x$ we denote by $(s^{U}_x)$, $(s^V_x)$ the column matrices of components of the section $s_x$ with respect to $\mathcal{B}_U$ and $\mathcal{B}_V$, respectively. Then $(s^V_x)=P\cdot(s^{U}_x)$, where
$$P=\left( 
\begin{array}{ll}
A\,& B\,\,\,\,\,\,\,\,0 \\ 
C\, & D\,\,\,\,\,\,\,\,0\\
0\, &\,0\,\,\,\,\,\,\,\,\,1
\end{array}
\right)$$ 
and $A,B,C,D\in\mathcal{M}_{m\times m}(\mathbb{R})$. But $P$ is orthogonal and commutes with the matrix $(F_E)$ (see Proposition \ref{pII2} (ii)), thus we have $D=A$, $C=-B$ and this proves that $P\in U(m)\times 1$.

Conversely, if the structure group of the vector bundle $E$ reduces to $U(m)\times 1$ then there exists a covering $\{U_\alpha\}_{\alpha\in I}$ of $M$, for which we can choose the orthonormal local bases of sections of $E$ with the property that on the intersection $U_\alpha\cap U_\beta\neq\emptyset$ these are transformed by the action of the group $U(m)\times 1$. With respect to such bases we can define the endomorphism $F_E|_\alpha:\Gamma(E|_{U_\alpha})\rightarrow\Gamma(E|_{U_\alpha})$ by the matrix $(F_E)$. But $(F_E)$ commutes with $U(m)\times 1$, hence $\{F_E|_\alpha\}_{\alpha\in I}$ determine a global endomorphism $F_E:\Gamma(E)\rightarrow\Gamma(E)$. In a similar way the sections $\xi\in\Gamma(E)$ and $\eta\in\Omega^1(E)$ are globally defined by the matrices of their components with respect to each open set $U_\alpha$, namely
\begin{displaymath}
\xi:(0,\ldots,0,1)^t\,,\,\eta:(0,\ldots,0,1).
\end{displaymath}
Finally, the fact that $(F_E,\xi,\eta)$ is an almost contact structure on $E$ is straighforward.
\end{proof}

Also, we notice that the determinants of the matrices from the proof of Theorem \ref{tII2} are positive which yields
\begin{corollary}
\label{c1}
Any almost contact bundle (or Lie algebroid) is orientable. 
\end{corollary}

From Proposition \ref{pII2} (iv) it follows that $\Omega_E$ defined by $\Omega_E(s_1,s_2)=g_E(s_1,F_E(s_2))$ for all $s_1,s_2\in\Gamma(E)$ is a $2$--form on $E$. It is called the \textit{fundamental $2$--form} or the \textit{Sasaki $2$--form} of the almost contact Riemaniann vector bundle (or Lie algebroid) $(E,F_E,\xi,\eta,g_E)$. Moreover, it is easy to see that $\Omega_E$ has the following obvious properties:
\begin{equation}
\label{II5}
\Omega_E(s_1,F_E(s_2))=-\Omega_E(F_E(s_1),s_2)\,\,\,{\rm and}\,\,\,\Omega_E(F_E(s_1),F_E(s_2))=\Omega_E(s_1,s_2).
\end{equation}
If  $\{e^{a},e^{a^*},\eta\}$ is the dual basis of the $F_E$--basis from Proposition \ref{pII2} then the fundamental $2$--form $\Omega_E$ is locally given by
\begin{displaymath}
\Omega_E=-2\sum_{a=1}^me^{a}\wedge e^{a^*}.
\end{displaymath}
We remark that ${\rm rank}\,\Omega_E=2m$ and then $\eta\wedge\Omega^m_E$ (where $\Omega_E^m$ is the exterior product of $m$ copies of $\Omega_E$) vanish nowhere on $M$. The converse of this result is also true, namely we have
\begin{theorem}
\label{tII3.1}
Let $E$ be a vector bundle over $M$ of ${\rm rank}\,E=2m+1$ and $\eta\in\Omega^1(E)$. If there exists $\Omega_E\in\Omega^2(E)$ such that $\eta\wedge\Omega_E^m\neq 0$ at each point of $M$ then $E$ has an almost contact structure.
\end{theorem}
\begin{proof}
Follows as in the case of almost contact manifolds (see for instance \cite{Bl1, Pi}).
\end{proof}
Moreover, in the case of Lie algebroids, we have 
\begin{theorem}
\label{tII3.2}
Let $(E,\rho_E,[\cdot,\cdot]_E)$ be a Lie algebroid of ${\rm rank}\,E=2m+1$ and $\eta\in\Omega^1(E)$. If $\eta\wedge(d_E\eta)^m\neq 0$ on $M$ then the Lie algebroid $(E,\rho_E,[\cdot,\cdot]_E)$ has an almost contact Riemannian structure $(F_E,\xi,\eta,g_E)$ whose fundamental form is $d_E\eta$ and the Reeb section $\xi$ is completely determined by the conditions $\eta(\xi)=1$ and $\imath_\xi(d_E\eta)=0$.
\end{theorem}

\section{Normal almost contact structures on Lie algebroids}
\setcounter{equation}{0}
In this section we define  normal almost contact structures on Lie algebroids and we characterize these structures. Also, the direct product between an almost Hermitian Lie algebroid with an almost contact Riemannian Lie algebroid or the direct product of two almost contact Riemannian Lie algebroids are investigated.

We recall that for a general tensor $A\in\Gamma(E\otimes E^*)$ of type $(1,1)$ on $E$, the Nijenhuis tensor of $A$ is a tensor $N_A\in\Gamma(\otimes^2E^*\otimes E)$ given by
\begin{displaymath}
N_A(s_1,s_2)=[A(s_1),A(s_2)]_E-A([A(s_1),s_2]_E)-A([s_1,A(s_2)]_E)+A^2([s_1,s_2]_E).
\end{displaymath}
As usual, we say that an almost contact structure $(F_E,\xi,\eta)$ on a Lie algebroid $(E,\rho_E,[\cdot,\cdot]_E)$ of rank $2m+1$ is \textit{normal} if 
\begin{equation}
\label{III1}
N_E^{(1)}\equiv N_{F_E}+2d_E\eta\otimes \xi=0.
\end{equation}
Other useful tensors on $E$ are the following:
\begin{equation}
\label{III2}
N_E^{(2)}(s_1,s_2)\equiv \left(\mathcal{L}_{F_E(s_1)}\eta\right)(s_2)-\left(\mathcal{L}_{F_E(s_2)}\eta\right)(s_1)\,,\,N_E^{(3)}(s)\equiv\frac{1}{2}\left(\mathcal{L}_{\xi}F_E\right)(s)\,,\,N_E^{(4)}(s)\equiv\left(\mathcal{L}_{\xi}\eta\right)(s).
\end{equation}
Using the differential calculus on Lie algebroids (exterior differential and Lie derivative),  we can easily prove that if the almost contact structure $(F_E,\xi,\eta)$ is normal then $N_E^{(2)}=N_E^{(3)}=N_E^{(4)}=0$.

Replacing in the definition of the Nijenhuis tensor $N_{F_E}$ the brackets by their expressions (since the Levi-Civita connection $\nabla$ on Riemannian Lie algebroids is torsionless, see \cite{Bo}), similarly to the almost contact Riemannian manifolds (see \cite{Ta65}), we obtain
\begin{proposition}
An almost contact Riemannian structure $(F_E,\xi,\eta,g_E)$ on a Lie algebroid $(E,\rho_E,[\cdot,\cdot]_E)$ is normal if and only if one of the following conditions is satisfied:
\begin{equation}
\label{III3}
F_E\left(\nabla_{s_1}F_E\right)s_2-\left(\nabla_{F_E(s_1)}F_E\right)s_2-\left[\left(\nabla_{s_1}\eta\right)s_2\right]\xi=0,
\end{equation}
\begin{equation}
\label{III4}
\left(\nabla_{s_1}F_E\right)s_2-\left(\nabla_{F_E(s_1)}F_E\right)F_E(s_2)+\eta(s_2)\nabla_{F_E(s_1)}\xi=0,
\end{equation}
for every $s_1,s_2\in\Gamma(E)$.
\end{proposition}
Since the eigenvalues of $F_E|_D$ are $i$ and $-i$, we deduce that the complexified $D_{\mathbb{C}}=D\otimes_{\mathbb{R}}\mathbb{C}$ of $D$ has the decomposition
\begin{equation}
\label{III5}
D_{\mathbb{C}}=D^{1,0}\oplus D^{0,1},
\end{equation}
where $D^{1,0}$ and $D^{0,1}$ are the eigensubbundles corresponding to $i$ and $-i$, respectively. A simple argument shows that
\begin{displaymath}
D^{1,0}=\{s-iF_E(s)\,|\,s\in\Gamma(D)\}\,,\,D^{0,1}=\{s+iF_E(s)\,|\,s\in\Gamma(D)\}
\end{displaymath}
and extending to $E_{\mathbb{C}}$ the metric $g_E$ by
\begin{displaymath}
g_E^c(s_1+is_2,s)=g_E(s_1,s)+ig_E(s_2,s)\,,\,g_E^c(s,s_1+is_2)=g_E(s,s_1)-ig_E(s,s_2)
\end{displaymath}
we obtain a Hermitian metric $g_E^c$ on $E_{\mathbb{C}}$. From Proposition \ref{pII2} (iv), we deduce that with respect to this metric the decomposition \eqref{III5} is orthogonal and to this one the following orthogonal decomposition of the complexified vector bundle $E_{\mathbb{C}}$ is associated
\begin{equation}
\label{III6}
E_{\mathbb{C}}=D_{\mathbb{C}}\oplus\langle \xi\rangle_{\mathbb{C}}=D^{1,0}\oplus D^{0,1}\oplus\langle \xi\rangle_{\mathbb{C}},
\end{equation}
where $\langle \xi\rangle_{\mathbb{C}}=\langle \xi\rangle\otimes_{\mathbb{R}}{\mathbb{C}}$.

On the other hand, $(E_{\mathbb{C}},g_E^c)$ is a Hermitian vector bundle over $M$ and the natural extension $\nabla^c$ of the Levi-Civita connection $\nabla$ from $E$ is a Hermitian connection in this bundle (see \cite{I-Po}). Moreover, $(D_{\mathbb{C}},g^c_E|_{D_{\mathbb{C}}})$ is a Hermitian subbundle of $(E_{\mathbb{C}},g^c_{E})$, with the Hermitian connection $\nabla^{D_{\mathbb{C}}}$ induced by the following decomposition
\begin{equation}
\label{III7}
\nabla^cs=\nabla^{D_{\mathbb{C}}}s+A^{D_{\mathbb{C}}}s,
\end{equation}
where $s\in\Gamma(D_{\mathbb{C}})$, $\nabla^{D_{\mathbb{C}}}s\in L(E_{\mathbb{C}},D_{\mathbb{C}})$ and $A^{D_{\mathbb{C}}}s\in L(E_{\mathbb{C}},\langle \xi\rangle_{\mathbb{C}})$. A simple calculation shows that
\begin{displaymath}
A^{D_{\mathbb{C}}}_{s}s^\prime=-\Omega_E(s,s^\prime)\xi\,\,,\,\,\,\nabla^{D_{\mathbb{C}}}F_E|_{D_{\mathbb{C}}}=0,
\end{displaymath}
hence $\nabla^{D_{\mathbb{C}}}$ is an almost complex connection, \cite{I-Po}, in the complex bundle $D_{\mathbb{C}}$.

Let $g_E^{1,0}$ be the restriction of the metric $g^c_E|_{D_{\mathbb{C}}}$ to $D^{1,0}$. Following the same argument as above we deduce that $(D^{1,0},g^{1,0}_E)$ is a Hermitian subbundle of $(D_{\mathbb{C}},g^c_E|_{D_{\mathbb{C}}})$, with Hermitian connection $\nabla^{1,0}$ induced by the following decomposition
\begin{equation}
\label{III8}
\nabla^{D_{\mathbb{C}}}s=\nabla^{1,0}s+A^{1,0}s,
\end{equation}
where $s\in\Gamma(D^{1,0})$, $\nabla^{1,0}s\in L(D_{\mathbb{C}},D^{1,0})$ and $A^{1,0}s\in L(D_{\mathbb{C}},D^{0,1})$.

The direct product of two Lie algebroids $\left(E_1,\rho_{E_1},[\cdot,\cdot]_{E_1}\right)$ over $M_1$ and $\left(E_2,\rho_{E_2},[\cdot,\cdot]_{E_2}\right)$ over $M_2$ is defined in \cite{Mack2} pg. 155, as a Lie algebroid structure on $E_1\times E_2\rightarrow M_1\times M_2$. The general sections of $E_1\times E_2$ are of the form $s=\sum (f_i\otimes s_i^1)\oplus\sum(g_j\otimes s_j^2)$, where $f_i,g_j\in C^\infty(M_1\times M_2)$, $s_i^1\in\Gamma(E_1)$, $s_j^2\in\Gamma(E_2)$, the anchor map is defined by
\begin{displaymath}
\rho_E\left(\sum (f_i\otimes s_i^1)\oplus\sum(g_j\otimes s_j^2)\right)=\sum(f_i\otimes\rho_{E_1}(s_i^1))\oplus\sum(g_j\otimes\rho_{E_2}(s_j^2)),
\end{displaymath}
and the Lie bracket on $E=E_1\times E_2$ is:
\begin{eqnarray*}
[s,s^\prime]_E&=&\left(\sum f_if^\prime_k\otimes[s_i^1,s_k^{\prime 1}]_{E_1}+\sum f_i\rho_{E_1}(s_i^1)(f_k^\prime)\otimes s_k^{\prime 1}-\sum f_k^\prime\rho_{E_1}(s_k^{\prime 1})(f_i)\otimes s_i^1\right)\\
&&\oplus\left(\sum g_jg^\prime_l\otimes[s_j^2,s_l^{\prime 2}]_{E_2}+\sum g_j\rho_{E_2}(s_j^2)(g_l^\prime)\otimes s_l^{\prime 2}-\sum g_l^\prime\rho_{E_2}(s_l^{\prime 2})(g_j)\otimes s_j^2\right)
\end{eqnarray*}
for every $s=\sum (f_i\otimes s_i^1)\oplus\sum(g_j\otimes s_j^2)$ and $s^{\prime}=\sum (f_k^{\prime}\otimes s_k^{\prime 1})\oplus\sum(g^{\prime}_l\otimes s_l^{\prime 2})$ in $\Gamma(E)$.

Now, by direct verification and using a simple calculation we can prove the following two results concerning the direct product of Lie algebroids.
\begin{proposition}
\label{pIII2}
Let us consider two Lie algebroids $\left(E_1,\rho_{E_1},[\cdot,\cdot]_{E_1}\right)$ over $M_1$ of rank $2m_1$ equipped with an almost Hermitian structure $(J_{E_1},g_{E_1})$, \cite{I-Po}, and $\left(E_2,\rho_{E_2},[\cdot,\cdot]_{E_2}\right)$ over $M_2$ of rank $2m_2+1$ equipped with an almost contact Riemannian structure $(F_{E_2},\xi_2,\eta_2,g_{E_2})$. Then the tensors $F_E$, $\xi$, $\eta$, $g_E$, given by
\begin{displaymath}
F_E\left(\sum(f_i\otimes s_i^1)\oplus\sum(g_j\otimes s_j^2)\right)=\sum(f_i\otimes J_{E_1}(s_i^1))\oplus\sum(g_j\otimes F_{E_2}(s_j^2)),
\end{displaymath}
\begin{displaymath}
\eta\left(\sum(f_i\otimes s_i^1)\oplus\sum(g_j\otimes s_j^2)\right)=\sum(g_j\otimes\eta_2(s_j^2))\,,\,\xi=0\oplus \xi_2,
\end{displaymath}
and 
\begin{displaymath}
g_E\left(\left(\sum(f_i\otimes s_i^1)\oplus\sum(g_j\otimes s_j^2)\right),\sum (f_k^{\prime}\otimes s_k^{\prime 1})\oplus\sum(g^{\prime}_l\otimes s_l^{\prime 2})\right)
\end{displaymath}
\begin{displaymath}
=\sum f_if_k^\prime\otimes g_{E_1}(s_i^1,s_k^{\prime 1})\oplus\sum g_jg_l^{\prime}\otimes g_{E_2}(s_j^2,s_{l}^{\prime 2})
\end{displaymath}
define an almost contact Riemannian structure on the direct product Lie algebroid $E=E_1\times E_2$.
\end{proposition}

\begin{proposition}
\label{pIII3}
Let us consider two Lie algebroids $\left(E_1,\rho_{E_1},[\cdot,\cdot]_{E_1}\right)$ over $M_1$ of rank $2m_1+1$ equipped with an almost contact Riemannian structure $(F_{E_1},\xi_1,\eta_1,g_{E_1})$ and $\left(E_2,\rho_{E_2},[\cdot,\cdot]_{E_2}\right)$ over $M_2$ of rank $2m_2+1$ equipped with an almost contact Riemannian structure $(F_{E_2},\xi_2,\eta_2,g_{E_2})$. Then the tensor $F_E$ given by
\begin{displaymath}
F_E\left(\sum(f_i\otimes s_i^1)\oplus\sum(g_j\otimes s_j^2)\right)=\sum(f_i\otimes F_{E_1}(s_i^1)-g_j\otimes\eta_2(s_j^2)\xi_1)\oplus\sum(g_j\otimes F_{E_2}(s_j^2)+f_i\otimes\eta_1(s_i^1)\xi_2),
\end{displaymath}
defines an almost Hermitian structure on the direct product Lie algebroid $E=E_1\times E_2$, with the metric $g_E$ from Proposition \ref{pIII2}. This structure is Hermitian (that is $N_{F_E}=0$) if and only if both almost contact Riemannian structures are normal.
\end{proposition}
\begin{remark}
Let $(E,F_E,\xi,\eta)$ be an almost contact Lie algebroid of rank $2m+1$ over a smooth manifold $M$ and $L$ be a line Lie algebroid over $M$ such that $\Gamma(L)=span\,\{s_L\}$. Then if we consider the Lie algebroid $\widetilde{E}$ given by direct product $\widetilde{E}=E\times L$, we remark that the map
\begin{displaymath}
J_{\widetilde{E}}:\Gamma(\widetilde{E})\rightarrow\Gamma(\widetilde{E})\,,\,J_{\widetilde{E}}(s\oplus fs_L)=(F_E(s)-f\xi)\oplus\eta(s)s_L
\end{displaymath} 
for every $f\in C^\infty(M)$, $s\in\Gamma(E)$ is linear and $J_{\widetilde{E}}^2=-I_{\widetilde{E}}$, that is $(\widetilde{E},J_{\widetilde{E}})$ is an almost complex Lie algebroid of rank $2m+2$. Also, as usual, we can prove that the almost contact structure $(F_E,\xi,\eta)$ on $E$ is normal if $J_{\widetilde{E}}$ is integrable.
\end{remark}

The following formula is useful for the calculation of the covariant derivative of $F_E$ depending on the tensors $N_E^{(1)}$ and $N_E^{(2)}$, in the case of arbitrary almost contact Riemannian structures on Lie algebroids.
\begin{proposition}
\label{pIII4}
Let $(F_E,\xi,\eta,g_E)$ be an almost contact Riemannian structure on the Lie algebroid $\left(E,\rho_{E},[\cdot,\cdot]_{E}\right)$ of rank $2m+1$ over a smooth manifold $M$. If $\nabla$ is the Levi-Civita connection of the metric $g_E$ then 
\begin{eqnarray*}
2g_E\left((\nabla_{s_1}F_E)s_2,s_3\right)&=&3d_E\Omega_E(s_1,F_E(s_2),F_E(s_3))-3d_E\Omega_E(s_1,s_2,s_3)+g_E(N_E^{(1)}(s_2,s_3),F_E(s_1))\\
& & +N_E^{(2)}(s_2,s_3)\eta(s_1)+2d_E\eta(F_E(s_2),s_1)\eta(s_3)-2d_E\eta(F_E(s_3),s_1)\eta(s_2)
\end{eqnarray*}
for every $s_1,s_2,s_3\in\Gamma(E)$.
\end{proposition}
\begin{proof}
Follows by direct calculus.
\end{proof}

\section{Contact structures on Lie algebroids}
\setcounter{equation}{0}
In this section we give the basic definitions and results about contact structures on Lie algebroids in relation with similar notions from contact manifolds theory, we present some examples from \cite{P-C-M1, P-C-M2}, we present a bijective corespondence between contact Riemannian structures and almost contact Riemannian structures on Lie algebroids, and we give some characterizations of contact Riemannian Lie algebroids. Also, the notions of $K$-contact, Sasakian and Kenmotsu Lie algebroids are introduced and some of their properties are analyzed.

\subsection{Contact Lie algebroids}

Firstly, we recall that a contact structure on an odd dimensional manifold $M^{2n+1}$ is defined by a maximally non-integrable distribution of rank $2n$, $\mathcal{D}^{2n}\subset TM$, called \textit{contact distribution}. Equivalently, we have that the curvature form of the distribution $\mathcal{D}^{2n}$ is non-degenerate. Moreover, if there exists a $1$-form $\eta\in\Omega^1(M)$ such that $\ker \eta=\mathcal{D}^{2n}$ then the contact structure is called \textit{cooriented}. Also, we notice that the contact structures on foliated manifolds was recently introduced (see for instance \cite{Ca-Pin-Pres, Pin-Pres}) as a triple $(M^{2n+1+m}, \mathcal{F}^{2n+1}, \mathcal{D}^{2n})$, where $M$ is a smooth manifold of dimension $2n+1+m$, $\mathcal{F}$ is a foliation of codimension $m$ ($\dim\mathcal{F}=2n+1$), and $\mathcal{D}\subset T\mathcal{F}$ is a distribution of dimension $2n$ (of the tangent bundle along the leaves) that is contact on each leaf of $\mathcal{F}$. This generalizes the contact fiber bundles construction from \cite{Le}. A standard example of foliated contact structure is the \textit{space of foliated oriented contact elements} on the cotangent spheric bundle $S(T^*\mathcal{F})$ of the leafwise cotangent bundle of $\mathcal{F}$ (see \cite{Pin-Pres}), which can be also obtained directly by pullback of the natural foliated contact structure on the projectivised cotangent bundle $P(T^*\mathcal{F})$ of $\mathcal{F}$ via the double-cover $S(T^*\mathcal{F})\rightarrow T^*\mathcal{F}\rightarrow P(T^*\mathcal{F})$ (see \cite{Ca-Pin-Pres}).  

These notions concerning to foliated contact structures can serve as elementary examples of our next general considerations, because it is well known that for a given regular foliated manifold  $(M,\mathcal{F})$, the tangent bundle  along the leaves $T\mathcal{F}$ has a natural structure of Lie algebroid with anchor the inclusion $i:T\mathcal{F}\rightarrow TM$ and the bracket the usual Lie bracket of vectors fields tangent to the leaves. Hence, the study of contact structures in the context of general Lie algebroids is natural and it can be of some interest.  

Let us continue with some basic definitions and results about contact Lie algebroids in relation with similar notions from contact manifolds/foliations theory . 

Let $\left(E,\rho_{E},[\cdot,\cdot]_{E}\right)$ be a Lie algebroid of rank $2m+1$ over a smooth manifold $M$. If a $1$--form $\eta$ on $E$, satisfying the condition from Theorem \ref{tII3.2}  is given, namely if $\eta\wedge (d_E\eta)^m\neq0$ everywhere on $M$, then we say that $\eta$ defines a \textit{contact structure} on $E$ or that $(E,\eta)$ is a \textit{contact Lie algebroid} and $\eta$ is called the \textit{contact form} of $E$. We remark that if $f\in C^\infty(M)$ nowhere vanishes on $M$ then $f\eta$ also is a contact form on $E$. Moreover, $\eta$ and $f\eta$ determine the same contact subbundle $D$, hence the authentic invariant of this change of contact forms is the contact subbundle. For this reason it is more natural to define a contact structure by a subbundle $D$ of rank $2m$ of $E$, with the property that there exists a $1$--form $\eta\in\Omega^1(E)$ so that $D=\cup_{x\in M} D_x$, where  $\ker\eta_x=D_x$ and $\eta\wedge(d_E\eta)^m$ nowhere vanishes on $M$. Alternatively, a \textit{contact structure on $E$} is given by a  pair $(\theta_E,\Omega_E)$, where $\theta_E\in\Omega^1(E)$ is a $1$-form on $E$ and $\Omega_E\in\Omega^2(E)$ is a $2$-form on $E$ such that $\Omega_E=d_E\theta_E$ and $(\theta_E\wedge\Omega_E\wedge\ldots\stackrel{m }{}\ldots\wedge\Omega_E)(x)\neq0$ , for every $x\in M$. The Reeb section $R\in\Gamma(E)$ is defined by $\imath_R\theta_E=1$ and $\imath_R\Omega_E=0$.

\begin{example}
\label{spheric} 
(\cite{P-C-M1}) For a Lie algebroid $(E,[\cdot,\cdot]_E,\rho_E)$  of rank $m$ over $M$ we can consider the prolongation of $E$ over its dual bundle $p^*:E^*\rightarrow M$ (see \cite{H-M, L-M-M}), which is a vector bundle $(\mathcal{T}^{E}E^*,p^*_1, E^*)$, where $\mathcal{T}^{E}E^*=\cup_{u^*\in E^*}\mathcal{T}^{E}_{u^*}E^*$ with
\begin{displaymath}
\mathcal{T}^{E}_{u^*}E^*=\left\{(u_x,V_{u^*})\in E_x\times T_{u^*}E^*\,|\,\rho_E(u_x)=(p^*)_*(V_{u^*})\,,\,p^*(u^*)=x\in M\right\},
\end{displaymath}
and the projection $p^*_1:\mathcal{T}^{E}E^*\rightarrow E^*$ given by $p^*_1(u_x,V_{u^*})=u^*$. A section $\widetilde{s}\in\Gamma(\mathcal{T}^{E}E^*)$ is called projectable if and only if there exist $s\in\Gamma(E)$ and $V\in\mathcal{X}(E^*)$ such that $(p^*)_*(V)=\rho_E(s)$ and $\widetilde{s}=((s(p^*(u^*)),V(u^*))$. We notice that $\mathcal{T}^{E}E^*$ has a Lie algebroid structure of rank $2m$ over $E^*$ with anchor $\rho_{\mathcal{T}^{E}E^*}:\mathcal{T}^{E}E^*\rightarrow TE^*$ given by $\rho_{\mathcal{T}^{E}E^*}(u,V)=V$ and Lie bracket uniquely determined by
\begin{displaymath}
[(s_1,V_1),(s_2,V_2)]_{\mathcal{T}^{E}E^*}=([s_1,s_2]_E,[V_1,V_2])\,,\,\,s_1,s_2\in\Gamma(E),\,V_1,V_2\in\mathcal{X}(E^*).
\end{displaymath}
The Liouville section $\lambda_E\in\Gamma((\mathcal{T}^{E}E^*)^*)$ is given by $\lambda_E(u^*)(u,V)=u^*(u)$, $u^*\in E^*$, $(u,V)\in \mathcal{T}^{E}E^*$, and the canonical symplectic section $\omega_E\in\Omega^2(\mathcal{T}^{E}E^*)$ is given by $\omega_E=-d_{\mathcal{T}^{E}E^*}\lambda_E$, thus $(\mathcal{T}^{E}E^*,\omega_E)$ is a symplectic Lie algebroid.

Now, we suppose that we have a bundle metric $g_E$ on $E$ and we consider the associated spherical bundle $p_{S(E^*)}:S(E^*)\rightarrow M$ of rank $m-1$, where $S(E^*)=\{u^*\in E^*\,|\,g_E^*(u^*,u^*)=1\}$.

Similarly as above we can consider the prolongation $\mathcal{T}^{E}S(E^*)$ of $E$ over the spherical bundle $S(E^*)$, and for the following diagram
$$
\begin{CD}
\mathcal{T}^{E}S(E^*)  @>\mathcal{T}_Ei>> \mathcal{T}^{E}E^*\\
@ V\tau_{\mathcal{T}^{E}S(E^*)}VV   @VV\tau_{\mathcal{T}^{E}E^*}=p_1^*V\\
S(E^*)  @>i>> E^*
\end{CD}
$$
we have
$d_{\mathcal{T}^{E}S(E^*)}((\mathcal{T}_Ei)^*\varphi)=(\mathcal{T}_Ei)^*(d_{\mathcal{T}^{E}E^*}\varphi)\,,\,\,\varphi\in\Omega(\mathcal{T}^{E}E^*)$, that is $\mathcal{T}^{E}S(E^*)\rightarrow S(E^*)$ is a Lie subalgebroid of $\mathcal{T}^{E}E^*\rightarrow E^*$. 

Now, for $\eta_E=-(\mathcal{T}_Ei)^*(\lambda_E)\in\Omega^1(\mathcal{T}^{E}S(E^*))$ we have $\eta_E\wedge(d_{\mathcal{T}^{E}S(E^*)}\eta_E)^{m-1}\neq0$, that is $(\mathcal{T}^{E}S(E^*),\eta_E)$ is a contact Lie algebroid.
\end{example}
\begin{remark}
More generally if $(E,[\cdot,\cdot]_E,\rho_E)$ is an exact symplectic Lie algebroid over $M$ of rank $2m$ with exact symplectic section $\Omega=-d_E\lambda$ and $F\rightarrow N$ is a Lie subalgebroid of rank $2m-1$ of $E$ then according to \cite{P-C-M1, P-C-M2}, $(F,\eta=i_F^*(\lambda))$ is a contact Lie algebroid, where $i_F:F\rightarrow E$ is the natural inclusion.
\end{remark}
The above definition is so called cooriented contact stucture and $\eta$ such that $\ker\eta=D$ is called a coorientation of the contact structure $(E,D)$. However, as in the case of smooth manifolds (see for instance \cite{C-S}), we can talk about general contact structures on Lie algebroids (not necessarily cooriented) and their brackets as follows.
\begin{definition}
A \textit{contact structure} on a Lie algebroid $(E,[\cdot,\cdot]_E,\rho_E)$ of rank $2m+1$ is a subbundle $D$ of rank $2m$ of  $E$ which is maximally non-integrable, that is, the curvature $Curv(D):D\times D\rightarrow L$ is non-degenerate, where $L$ is the quotient line bundle $L:=E/D$ and $Curv(D)$ is given at the level of sections by $Curv(D)(s_1,s_2)=[s_1,s_2]_E$mod$D$. The pair $(E,D)$ is called a \textit{contact Lie algebroid}.
 \end{definition}
\begin{definition}
A Reeb section of the contact Lie algebroid $(E,D)$ is every section $\xi\in\Gamma(E)$ such that $[\xi,\Gamma(D)]_E\subset\Gamma(D)$, and we denote by $\Gamma_{{\rm Reeb}}(E,D)$ the set of Reeb sections.
\end{definition}
\begin{proposition}
\label{lareeb}
The set of Reeb sections of a contact Lie algebroid $(E,D)$ is a Lie subalgebra of the Lie algebra $\Gamma(E)$ of all sections of $E$ and $\Gamma(E)=\Gamma_{{\rm Reeb}}(E,D)\oplus\Gamma(D)$.
\end{proposition}
\begin{proof}
Follows as in the contact manifolds case (see \cite{C-S}).
\end{proof}
Also, it is useful to consider the dual point of view on contact structures on Lie algebroids, that is to view $D$ as the kernel of a $1$-form  on $E$ with values in $L$ ($\theta_E\in\Omega^1(E,L)$ and viewed as the canonical projection from $E$ to $L$). Now, the curvature of $D$ can be rewritten as $Curv(D)(s_1,s_2)=\theta_E\left([s_1,s_2]_E\right)$, and we say that $\theta_E$ is of \textit{contact type}. The case when $L$ is the trivial line bundle gives rise to the  above cooriented case. The previous proposition yields
\begin{corollary}
The $1$-form  $\theta_E$ with values in $L$ restricts to a vector space isomorphism
\begin{equation}
\label{iso-Reeb}
\theta_E|_{\Gamma_{{\rm Reeb}}(E,D)}:\Gamma_{{\rm Reeb}}(E,D)\stackrel{\cong}{\rightarrow}\Gamma(L).
\end{equation}
\end{corollary}
Thus, the Lie algebra structure of $\Gamma_{{\rm Reeb}}(E,D)$ (from Proposition \ref{lareeb}) can be transfered to a Lie algebra structure on $\Gamma(L)$ and denote the corresponding bracket by $\{\cdot,\cdot\}_{L}$.
\begin{definition}
The bracket $\{\cdot,\cdot\}_{L}$ on $\Gamma(L)$ is called the \textit{Reeb bracket} associated to the contact Lie algebroid $(E,D)$ (which is a Kirillov type bracket \cite{Kir}).
\end{definition}
Also, we notice that similarly as in the contact manifolds case (see Lemma 2.5 from \cite{C-S}), the Proposition \ref{lareeb} can be reformulated in the form
\begin{proposition}
\label{prop-hom}
The map $\Gamma(E)\cong\Gamma(L)\oplus\Gamma({\rm Hom}(D,L))$, $s\mapsto(\theta_E(s),\theta_E\left([\cdot,s]_E\right))$, is an isomorphism of vector spaces and the induced $C^\infty(M)$-module structure on the right hand side is given by $f\cdot(s,\phi)=(fs,\phi+d_Ef\otimes s)$, for every $s\in\Gamma(L)$ and $\phi\in\Gamma({\rm Hom}(D,L))$.
\end{proposition}
The surjectivity of \eqref{iso-Reeb} implies that for every section $s\in\Gamma(L)$, there exists a unique section $\xi_s\in\Gamma(E)$ such that $\theta_E(\xi_s)=s$ and $\theta_E\left([\xi_s,t]\right)=0$ for every section $t\in\Gamma(D)$. In this case $\xi_s$ is called the \textit{Reeb section associated to $s$} and the Reeb bracket $\{\cdot,\cdot\}_L$ has the following characteristic property: $[\xi_{s_1},\xi_{s_2}]_E=\xi_{\{s_1,s_2\}_L}$, for every $s_1,s_2\in\Gamma(L)$. Moreover, applying $\theta_E$, we get the explicit formula for the Reeb bracket in terms of the $1$-form $\theta_E$, namely
\begin{equation}
\label{Reeb-bracket}
\{s_1,s_2\}_L=\theta_E\left([\xi_{s_1},\xi_{s_2}]_E\right), \,s_1,s_2\in\Gamma(L).
\end{equation}
The Proposition \ref{prop-hom} implies that, for $f\in C^\infty(M)$ and $s\in\Gamma(L)$, we have
\begin{equation}
\label{eq-reeb}
\xi_{fs}=f\xi_s+\beta(d_Ef\otimes s),
\end{equation}
where $\beta:{\rm Hom}(D,L)\rightarrow D$ is the isomorphism induced by $Curv(D)$, that is $${\rm Hom}(D,L)\ni Curv(D)(t,\cdot)\mapsto t\in\Gamma(D).$$
Also, we notice that the inverse of the isomorphism defined in Proposition \ref{prop-hom} sends $(s,\phi)$ to $\xi_s-\beta(\phi)$.
\begin{example}
When $L$ is the trivial line bundle the Reeb section associated to the constant function $1$ is the standard Reeb section $\xi$ associated to the contact form $\eta$ and it is uniquely determined by $\imath_\xi\eta=1$ and $\imath_\xi(d_E\eta)=0$. The other Reeb section corresponding to an arbitrary smooth function $f\in C^\infty(M)$ is $\xi_f=f\xi+\beta(d_Ef)$. In this case $\beta:D^*\rightarrow D$ is the isomorphism induced by $d_E\eta$. Finally, we notice that the Reeb bracket becomes a Jacobi bracket on $C^\infty(M)$ as follows
\begin{equation}
\label{bracket}
\{f,g\}_L=\Lambda(d_Ef,d_Eg)+\rho_E(\xi)(f)g-f\rho_E(\xi)(g),
\end{equation} 
where the bisection $\Lambda\in\Gamma(\bigwedge^2E)$ is defined by using $\beta$, that is $\Lambda(d_Ef,d_Eg)=d_E\eta(\beta(d_Ef),\beta(d_Eg))$. 
\end{example}

\begin{remark}
In some recent papers (see \cite{Br-Gr-Gr, Gra}) are introduced contact structures on principal $\mathbb{R}^\times:=\mathbb{R}-\{0\}$-bundles (using a new language about contact structures). More exactly, for a given $\mathbb{R}^\times$-action $h:\mathbb{R}^\times\times P\rightarrow P$ on a vector bundle $P\rightarrow M$, a \textit{contact structure} is referred as a triple $(P,h,\omega)$ where $\omega$ is a $1$-homogeneous symplectic form on $P$, that is $(h_t)^*\omega=t\omega\,(t\neq0)$. Using a similar language, the construction from \cite{P-C-M1} (recalled in Example \ref{spheric}) can be formulated in the non coorientable case as follows. Let $(E,[\cdot,\cdot]_E,\rho_E)$  be a Lie algebroid of rank $m$ over $M$, $p^*:E^*\rightarrow M$ the dual vector bundle of $E$ and $h:\mathbb{R}^\times\times E^*\rightarrow E^*$ be the multiplicative $\mathbb{R}^\times$-action on $E^*$ (then the projective bundle of $E^*$ is $P(E^*):=E^*/\mathbb{R}^\times\rightarrow M$, rank$P(E^*)=m-1$). As usual (for tangent and cotangent lifts of a $\mathbb{R}^\times$-action on manifolds or supermanifolds \cite{Gra}) there is a natural lift of $h$ to a $\mathbb{R}^\times$-action on $\mathcal{T}^EE^*$ denoted by $\mathcal{T}^Eh:\mathbb{R}^\times\times\mathcal{T}^EE^*\rightarrow\mathcal{T}^EE^*$ given by $(\mathcal{T}^Eh)_t=\mathcal{T}^E(h_t)$, which is a compatible action, that is $(\mathcal{T}^Eh)_t$ are Lie algebroids automorphisms (see \cite{Mart}). Then, there is a natural Lie algebroid (over quotied spaces) $P(\mathcal{T}^EE^*)\rightarrow P(E^*)$ which is isomorphic with the prolongation $\mathcal{T}^EP(E^*)$ of $E$ over the projective bundle $P(E^*)\rightarrow M$. Now, since the canonical symplectic section $\omega_E\in\Omega^2(\mathcal{T}^EE^*)$ is linear, thus homogeneous, the triple $(\mathcal{T}^EE^*, \mathcal{T}^Eh,\omega_E)$ is a contact structure. In a traditional language it corresponds to a contact structure on the Lie algebroid $\mathcal{T}^EP(E^*)\rightarrow P(E^*)$, let us say a maximally non-integrable subbundle   $\mathcal{D}^EP(E^*)\subset \mathcal{T}^EP(E^*)$ of rank $2m-2$, and then, the contact structure $\mathcal{D}^EP(E^*)$ pullback to a contact structure on the Lie algebroid $\mathcal{T}^ES(E^*)$ through the double-cover $S(E^*)\rightarrow E^*\rightarrow P(E^*)$.

\end{remark}

\subsection{Contact Riemannian Lie algebroids}
In what follows we consider only the coorientable case. When an almost contact Riemannian structure defined in Theorem \ref{tII3.2}  is fixed on the contact Lie algebroid $(E,\eta)$ then we say that $(E,F_E,\xi,\eta,g_E)$ is a \textit{contact Riemannian Lie algebroid}.
\begin{remark}
From the definition of the fundamental form and from Theorem \ref{tII3.2}  it results that for a given contact Riemannian structure, the endomorphism $F_E$ is uniquely determined by the $1$--form $\eta$ and by the metric $g_E$.
\end{remark}
For the contact Riemannian Lie algebroid $(E,F_E,\xi,\eta,g_E)$ we consider the contact subbundle $D$. Taking into account Theorem \ref{tII3.2}, the restriction to $D$ of the $2$--form $d_E\eta$ is nondegenerate and then we can state the following:
\begin{proposition}
\label{pIV1}
The contact subbundle $D$ of a contact Riemannian Lie algebroid has a symplectic vector bundle structure with the symplectic $2$--form $d_E\eta|_D$.
\end{proposition}
Denote by $\mathcal{J}(D)$ the set of almost complex structures on $D$, compatible with $d_E\eta$, that is the structures $\mathcal{J}:D\rightarrow D$ with the properties
\begin{equation}
\label{IV1}
\mathcal{J}^2=-I_D\,,\,d_E\eta(\mathcal{J}(s_1),\mathcal{J}(s_2))=d_E\eta(s_1,s_2)\,,\,d_E\eta(\mathcal{J}(s),s)\geq0
\end{equation}
for every $s,s_1,s_2\in\Gamma(D)$. This means that we consider on $D$ only almost complex structures compatible with its symplectic bundle structure. We remark that if $(F_E,\xi,\eta,g_E)$ is the almost contact Riemannian structure associated to the contact Riemannian structure defined in Theorem \ref{tII3.2} on the Lie algebroid $E$ then $F_E|_D\in\mathcal{J}(D)$.

For each $\mathcal{J}\in\mathcal{J}(D)$ the map $g_{\mathcal{J}}$, defined by
\begin{equation}
\label{IV2}
g_{\mathcal{J}}(s_1,s_2)=d_E\eta(\mathcal{J}(s_1),s_2)\,,\,\,s_1,s_2\in\Gamma(D)
\end{equation}
is a Hermitian metric on $D$, that is, it satisfies the condition
\begin{equation}
\label{IV3}
g_{\mathcal{J}}(\mathcal{J}(s_1),\mathcal{J}(s_2))=g_{\mathcal{J}}(s_1,s_2)\,,\,\,s_1,s_2\in\Gamma(D).
\end{equation}
Moreover, if we denote by $\mathcal{G}(D)$ the set of all Riemannian metrics on $D$, satisfying the equality \eqref{IV3}, it is easy to see that the map $\mathcal{J}\in\mathcal{J}(D)\mapsto g_{\mathcal{J}}\in\mathcal{G}(D)$ is bijective. Since $\eta$ nowhere vanishes on $M$, we denote by $\xi$ a section of $E$ such that $\eta(\xi)=1$ and extend $\mathcal{J}$ to an endomorphism $F_E$ of $\Gamma(E)$ by setting $F_E|_D=\mathcal{J}$, $F_E(\xi)=0$. Consider the decompositions $s_1=s_1^D+a\xi$, $s_2=s_2^D+b\xi$, where $s_1^D,s_2^D$ are the $D$ components of the sections $s_1$ and $s_2$, respectively. Similarly, we extend $g_{\mathcal{J}}$ to a metric on $E$ by
\begin{equation}
\label{IV4}
g_E(s_1,s_2)=g_{\mathcal{J}}(s_1^D,s_2^D)+ab
\end{equation}
for every $s_1,s_2\in\Gamma(E)$. Taking into account \eqref{IV2} we can prove that $d_E\eta(s_1,s_2)=g_E(s_1,F_E(s_2))$, hence the contact structure on $E$ is a Riemannian one. Moreover, $(F_E,\xi,\eta,g_E)$ is an almost contact Riemannian structure on $E$ and then the set of almost contact Riemannian structures on $E$ is in bijective correspondence with the set of almost complex structures of Hermitian type $(\mathcal{J},g_{\mathcal{J}})$ defined on the contact subbundle $D$.

Using the notion of a Killing section on Riemannian Lie algebroids (introduced recently in \cite{Br}) and the classical calculus on Lie algebroids, similar arguments used in the study of contact Riemannian manifolds (see \cite{Bl1,Pi}) yields
\begin{proposition}
\label{pIV2}
Let $E$ be a contact Riemannian Lie algebroid and let $(F_E,\xi,\eta,g_E)$ be the associated almost contact Riemannian structure. Then:
\begin{enumerate}
\item[(i)] $N_E^{(2)}=0$, $N_E^{(4)}=0$;
\item[(ii)] $N_E^{(3)}=0$ if and only if $\xi$ is a Killing section, i.e. $\mathcal{L}_\xi g_E=0$;
\item[(iii)] $\nabla_\xi F_E=0$.
\end{enumerate}
\end{proposition}

A more suitable form of the results from Proposition \ref{pIV2} is the following
\begin{proposition}
\label{pIV3}
Let $E$ be a contact Riemannian Lie algebroid and let $(F_E,\xi,\eta,g_E)$ be the associated almost contact Riemannian structure. Then:
\begin{displaymath}
\mathcal{L}_\xi\eta=0\,,\,\mathcal{L}_\xi(d_E\eta)=0\,,\,\left(\mathcal{L}_{F_E(s_1)}\eta\right)(s_2)=\left(\mathcal{L}_{F_E(s_2)}\eta\right)(s_1)
\end{displaymath}
for every $s_1,s_2\in\Gamma(E)$.
\end{proposition}
Another useful result in relation with corresponding notions from contact Riemannian manifolds is
\begin{proposition}
\label{pIV4}
On a contact Riemannian Lie algebroid the following formulas hold:
\begin{enumerate}
\item[(i)] $g_E(N_E^{(3)}(s_1),s_2)=g_E(s_1,N_E^{(3)}(s_2))$;
\item[(ii)] $\nabla_s\xi=-F_E(s)-F_E(N_E^{(3)}(s))$;
\item[(iii)] $F_E\circ N_E^{(3)}=-N_E^{(3)}\circ F_E$;
\item[(iv)] $trace\,N_E^{(3)}=0$, $trace\,(N_E^{(3)}\circ F_E)=0$, $N_E^{(3)}(\xi)=0$, $\eta(N_E^{(3)}(s))=0$;
\item[(v)] $(\nabla_{s_1}F_E)(s_2)+(\nabla_{F_E(s_1)}F_E)F_E(s_2)=2g_E(s_1,s_2)\xi-\eta(s_2)\left(s_1+N_E^{(3)}(s_1)+\eta(s_1)\xi\right)$.
\end{enumerate}
\end{proposition}

Now, by putting into other words Theorem \ref{tII3.2}, we can assert that if $\eta$ defines a contact structure on the Lie algebroid $E$ then there exists an almost contact Riemannian structure $(F_E,\xi,\eta,g_E)$ with $\Omega_E=d_E\eta$ as fundamental form. Then, it is natural to ask what kind of relation can exists between the form $\eta\wedge(d_E\eta)^m$ and the volume form $dV_{g_E}=\sqrt{\det g_E}e^1\wedge\ldots\wedge e^{2m+1}$ of the Riemannian metric $g_E$ on $E$. More exactly, following step by step the proof from the case of contact Riemannian manifolds (see \cite{Bl1, Bl, Pi}), we have the following
\begin{theorem}
\label{tIV1}
Let $E$ be a contact Riemannian Lie algebroid of rank $2m+1$ with contact $1$--form $\eta$. The volume form with respect to the metric $g_E$ of $E$ is given by
\begin{equation}
\label{IV8}
dV_{g_E}=\frac{1}{2^mm!}\eta\wedge(d_E\eta)^m.
\end{equation}
\end{theorem}
A morphism $\mu:(E_1,\eta_1)\rightarrow(E_2,\eta_2)$ between two contact Lie algebroids over the same manifold $M$ is called a \textit{contact morphism} if there is $f\in C^\infty(M)$ nowhere zero on $M$ and such that
\begin{equation}
\label{IV15}
\mu^*\eta_2=f\eta_1.
\end{equation}
If $f\equiv 1$ the morphism $\mu$ is called a \textit{strict contact morphism}. Also, we easily obtain
\begin{proposition}
\label{pIV5}
The morphism $\mu:(E_1,\eta_1)\rightarrow(E_2,\eta_2)$ between two contact Lie algebroids over the same manifold $M$ is a contact morphism if and only if $\mu(D_1)\subseteq D_2$.
\end{proposition}

\subsection{$K$--contact, Sasakian and Kenmotsu Lie algebroids}

A contact Riemannian Lie algebroid with the property that its Reeb section $\xi$ is a Killing section is called a \textit{$K$--contact Lie algebroid}. From Propositions \ref{pIV2} (ii) and \ref{pIV4} (ii) it easily follows
\begin{proposition}
\label{pIV6}
A contact Riemannian Lie algebroid $E$ is $K$--contact if and only if
\begin{equation}
\label{IV16}
\nabla_s\xi=-F_E(s)
\end{equation}
for every $s\in\Gamma(E)$.
\end{proposition}

From the formula \eqref{IV16} it results
\begin{proposition}
\label{pIV7}
On a $K$--contact Lie algebroid $E$ the following equalities hold
\begin{equation}
\label{IV17}
(\nabla_{s_1}\eta)s_2=g_E(\nabla_{s_1}\xi,s_2)=\Omega_E(s_1,s_2)\,,\,(\nabla_{s}F_E)\xi=-s+\eta(s)\xi
\end{equation}
for every $s,s_1,s_2\in\Gamma(E)$.
\end{proposition}

The contact Riemannian Lie algebroid $E$ is called a \textit{Sasakian Lie algebroid} if the associated almost contact Riemannian structure $(F_E,\xi,\eta,g_E)$ is normal. Otherwise, the almost contact Riemannian structure $(F_E,\xi,\eta,g_E)$ is a \textit{Sasakian structure} if $d_E\eta=\Omega_E$ and $N_E^{(1)}=0$.

From \eqref{III1} and Proposition \ref{pIV2} (ii) easily follows
\begin{theorem}
\label{tIV2}
Every Sasakian Lie algebroid is $K$--contact.
\end{theorem} 
A characterization of Sasakian Lie algebroids by the Levi-Civita connection $\nabla$ of $g_E$ can be obtained as in the manifolds case (see \cite{Bl1, Pi}), that is
\begin{theorem}
\label{tIV3}
The almost contact Riemannian structure $(F_E,\xi,\eta,g_E)$ on $E$ is Sasakian if and only if
\begin{equation}
\label{IV18}
(\nabla_{s_1}F_E)s_2=g_E(s_1,s_2)\xi-\eta(s_2)s_1
\end{equation}
for every sections $s_1,s_2\in\Gamma(E)$.
\end{theorem}

Choosing an $F_E$--basis $\{e_a\}=\{s_a,s_{a^*},\xi\}$ on $\Gamma(E)$, from \eqref{IV16} it follows
\begin{equation}
\label{IV21}
(\nabla_{e_a}\eta)e_b=g_E(\nabla_{e_a}\xi,e_b)=-g_E(F_E(e_a),e_b)=0.
\end{equation}
Now, using the $\star$--Hodge operator on invariantly oriented Lie algebroids (see \cite{Ba1}), the exterior coderivative on Lie algebroids can be expressed as
\begin{equation}
\label{IV22}
d_E^*\varphi=-\sum_{a=1}^{2m+1}\imath_{e_a}(\nabla_{e_b}\varphi)\,,\,\,\varphi\in\Omega^\bullet(E).
\end{equation}
Thus, from \eqref{IV21} and \eqref{IV22} we deduce $d_E^*\eta=0$, hence we can state the following
\begin{proposition}
\label{pIV8}
The contact form of a $K$--contact Lie algebroid is co-closed.
\end{proposition}
\begin{remark}
Assuming that the elements of the basis $\{e_a\}$ are eigensections of the operator $N_E^{(3)}$, by a similar argument it follows that Proposition \ref{pIV8} is valid for every contact Riemannian Lie algebroid.
\end{remark}
\begin{proposition}
\label{pIV9}
Every $K$--contact Lie algebroid of rank $3$ is Sasakian.
\end{proposition}
\begin{proof}
Denote by $\{e,F_E(e),\xi\}$ a $F_E$--basis of $\Gamma(E)$. Then we have
\begin{displaymath}
g_E((\nabla_{s}F_E)e,e)=0\,,\,g_E((\nabla_sF_E)e,F_E(e))=0\,,\,g_E((\nabla_sF_E)e,\xi)=g_E(s,e).
\end{displaymath}
We deduce $(\nabla_sF_E)e=g_E(s,e)\xi$ for every $s\in\Gamma(E)$ and then \eqref{IV18} is satisfied for $s_2=e$. Similarly one can verify \eqref{IV18} for $s_2=F_E(e)$ and $s_2=\xi$, hence by Theorem \ref{tIV3} the $K$--contact Lie algebroid of rank $3$ is Sasakian. 
\end{proof}

A Lie algebroid $(E,\rho_E,[\cdot,\cdot]_E)$ of ${\rm rank}\,E=2m+1$ endowed with an almost contact Riemannian structure $(F_E,\xi,\eta,g_E)$ is called an \textit{almost Kenmotsu Lie algebroid} if the following conditions are satisfied
\begin{equation}
\label{IV23}
d_E\eta=0\,,\,\,d_E\Omega_E=2\eta\wedge\Omega_E.
\end{equation}
We call a \textit{Kenmotsu Lie algebroid} every normal almost Kenmotsu Lie algebroid.
\begin{theorem}
\label{tIV4}
A Lie algebroid $(E,\rho_E,[\cdot,\cdot]_E)$ of ${\rm rank}\,E=2m+1$ endowed with an almost contact Riemannian structure $(F_E,\xi,\eta,g_E)$ is a Kenmotsu Lie algebroid if and only if
\begin{equation}
\label{IV24}
(\nabla_{s_1}F_E)s_2=-\eta(s_2)F_E(s_1)-g_E(s_1,F_E(s_2))\xi.
\end{equation}
\end{theorem}
\begin{proof}
Follows as in the case of Kenmotsu manifolds (see \cite{Pi}).
\end{proof}
Also, by straightforward calculation it follows
\begin{proposition}
\label{pIV10}
On a Kenmotsu Lie algebroid the following equalities hold:
\begin{displaymath}
(\nabla_{s_1}\eta)(s_2)=g_E(s_1,s_2)-\eta(s_1)\eta(s_2)\,,\,\mathcal{L}_\xi g_E=2(g_E-\eta\otimes\eta)\,,\,\mathcal{L}_\xi F_E=0\,,\,\mathcal{L}_\xi\eta=0.
\end{displaymath}
\end{proposition}
From Proposition \ref{pIV10} it follows that the Reeb section $\xi$ of a Kenmotsu Lie algebroid cannot be Killing, hence such Lie algebroid cannot be Sasakian and more generally, it cannot be $K$--contact.

\section{An almost contact Lie algebroid structure on the vertical Liouville distribution on the big-tangent manifold}
\setcounter{equation}{0}
The following definition which generalizes the notion of framed $f(3,1)$-structure from manifolds to Lie algebroids will be important for our next considerations.
\begin{definition}
A framed $f(3,1)$-structure of corank $s$ on a Lie algebroid $(E,\rho_E,[\cdot,\cdot]_E)$ of rank $(2n+s)$ is a natural generalization of an almost contact structure on $E$ and it is a triplet $(f, (\xi_{a}), (\omega^{a})), a =1,\ldots,s$, where $f\in\Gamma(E\otimes E^*)$ is a tensor of type $(1,1)$, $(\xi_{a})$ are sections of $E$ and $(\omega^{a})$ are $1$-forms on $E$ such that
\begin{equation}
 \omega^{a}(\xi_{b})=\delta^{a}_{b}\,,\,f(\xi_{a})=0\,,\,\omega^{a}\circ f=0\,,\,f^{2}=-I_E+\sum_{a}\omega^{a}\otimes\xi_{a}.
 \label{IIIx1}
 \end{equation}
\end{definition}
The name of $f(3,1)$-structure was suggested by the identity $f^{3}+f=0$. For an account of such kind of structures on manifolds we refer for instance to \cite{G-Y, Y}.

In this section we introduce a natural framed $f(3,1)$-structure of corank $2$ on the Lie algebroid defined by the vertical bundle over the big-tangent manifold of a Riemannian manifold $(M,g)$. When we restrict it to an integrable  vertical Liouville distribution over  the big-tangent manifold, which has a natural structure of Lie algebroid, we obtain an almost contact structure.

\subsection{Vertical framed $f$-structures on the big-tangent manifold}

The aim of this subsection is to construct some framed $f(3,1)$-structures on the vertical bundle $V=V_1\oplus V_2$ over the big-tangent manifold $\mathcal{T}M$ when $(M,g)$ is a Riemannian manifold.

Let $M$ be a $n$-dimensional smooth manifold, and we consider $\pi:TM\rightarrow M$ its tangent bundle, $\pi^*:T^*M\rightarrow M$ its cotangent bundle and $\tau\equiv \pi\oplus\pi^*:TM\oplus T^*M\rightarrow M$ its big-tangent bundle defined as the Whitney sum of the tangent and the cotangent bundles of $M$. The total space of the big-tangent bundle, called \textit{big-tangent manifold}, is a $3n$-dimensional smooth manifold denoted here by $\mathcal{T}M$. Let us briefly recall some elementary notions about the big-tangent manifold $\mathcal{T}M$. For a detalied discussion about its geometry we refer \cite{Va13}. 

Let $(U,(x^{i}))$ be a local chart on $M$. If $\{\frac{\partial}{\partial x^{i}}|_x\}$, $x\in U$ is a local frame of sections of the tangent bundle over $U$ and $\{dx^{i}|_x\}$, $x\in U$ is a local frame of sections of the cotangent bundle over $U$, then by definition of the Whitney sum, $\{\frac{\partial}{\partial x^{i}}|_x,dx^{i}|_x\}$, $x\in U$ is a local frame of sections of the big-tangent bundle $TM\oplus T^*M$ over $U$. Every section $(y,p)$ of $\tau$ over $U$ takes the form $(y,p)=y^{i}\frac{\partial}{\partial x^{i}}+p_i dx^{i}$ and the local coordinates on $\tau^{-1}(U)$ will be defined as the triples $(x^{i},y^{i},p_i)$,  where $i=1,\ldots,n=\dim M$, $(x^{i})$ are local coordinates on $M$, $(y^{i})$ are vector coordinates and $(p_i)$ are covector coordinates. 
The  local expressions of a vector field $X$ and of a $1$-form $\varphi$ on $\mathcal{T}M$ are
\begin{equation}
X=\xi^{i}\frac{\partial}{\partial x^{i}}+\eta^{i}\frac{\partial}{\partial y^{i}}+\zeta_i\frac{\partial}{\partial p_i}\,\,\,{\rm and}\,\,\,\varphi=\alpha_idx^{i}+\beta_idy^{i}+\gamma^{i}dp_i.
\label{I2}
\end{equation}
For the big-tangent manifold $\mathcal{T}M$ we have the following projections
\begin{displaymath}
\tau:\mathcal{T}M\rightarrow M\,,\,\tau_1:\mathcal{T}M\rightarrow TM\,,\,\tau_2:\mathcal{T}M\rightarrow T^*M
\end{displaymath}
on $M$ and on the total spaces of tangent and cotangent bundle, respectively. 
As usual, we denote by $V=V(\mathcal{T}M)$ the vertical bundle of the big-tangent manifold $\mathcal{T}M$ with respect to projection $\tau$ and it has the decomposition
\begin{equation}
\label{I4}
V=V_1\oplus V_2,
\end{equation}
where $V_1=\tau_1^{-1}(V(TM))$, $V_2=\tau_2^{-1}(V(T^*M))$ and have the local frames $\{\frac{\partial}{\partial y^{i}}\}$, $\{\frac{\partial}{\partial p_i}\}$, respectively. The subbundles $V_1$, $V_2$ are the vertical foliations of $\mathcal{T}M$ by fibers of $\tau_1, \tau_2$, respectively, and $\mathcal{T}M$ has a multi-foliate structure \cite{Va70}. The \textit{Liouville vector fields} are given by
\begin{equation}
\label{I5}
\mathcal{E}_1=y^{i}\frac{\partial}{\partial y^{i}}\in\Gamma(V_1)\,,\,\mathcal{E}_2=p_i\frac{\partial}{\partial p_i}\in\Gamma(V_2)\,,\,\mathcal{E}=\mathcal{E}_1+\mathcal{E}_2\in\Gamma(V).
\end{equation}

In the following we consider a Riemannian metric $g=(g_{ij}(x))_{n\times n}$ on the paracompact manifold $M$, and we put:
\begin{equation}
\label{f1}
y_i=g_{ij}y^j\,,\,p^i=g^{ij}p_j,
\end{equation}
where $(g^{ij})_{n\times n}$ denotes the inverse matrix of $(g_{ij})_{n\times n}$. It is well known that $g_{ij}$ determines in a natural way a Finsler metric on $TM$ by putting $F^2(x,y)=g_{ij}(x)y^iy^j$ and similarly, $g^{ij}$ determines a Cartan metric on $T^*M$ by putting $K^2(x,p)=g^{ij}(x)p_ip_j$. Then the relations \eqref{f1} imply
\begin{equation}
\label{f2}
y_iy^i=F^2\,,\,p_ip^i=K^2.
\end{equation}
Also, the Riemannian metric $g$ on $M$ determines a metric structure $G$ on $V$ by setting
\begin{equation}
\label{I8}
G(X,Y)=g_{ij}(x)X_1^{i}(x,y,p)Y_1^j(x,y,p)+g^{ij}(x)X^2_i(x,y,p)Y^2_j(x,y,p),
\end{equation}
for every $X=X_1^{i}(x,y,p)\frac{\partial}{\partial y^{i}}+X^2_i(x,y,p)\frac{\partial}{\partial p_i}$, $Y=Y_1^{j}(x,y,p)\frac{\partial}{\partial y^{j}}+Y^2_j(x,y,p)\frac{\partial}{\partial p_j}\in\Gamma(V)$.

Let us define the linear operator $\phi:V\rightarrow V$ given in the local vertical frames $\{\frac{\partial}{\partial y^i},\frac{\partial}{\partial p_i}\}$ by
\begin{equation}
\phi\left(\frac{\partial}{\partial y^i}\right)=-g_{ij}\frac{\partial}{\partial p_j}\,\,,\,\,\phi\left(\frac{\partial}{\partial p_i}\right)=g^{ij}\frac{\partial}{\partial y^j}.
\label{IIIx2}
\end{equation}
It is easy to see that $\phi$ defines an almost complex structure on $V$ and 
\begin{equation}
G(\phi(X),\phi(Y))=G(X,Y),\,\, \forall\,X,Y\in\Gamma(V).
\label{IIIx3}
\end{equation}
As $V$ is an integrable distribution on $\mathcal{T}M$ it follows that $(V,\phi,G)$ is a Hermitian Lie algebroid (foliation) over $\mathcal{T}M$ since $N_\phi=0$, where $N_\phi$ denotes the Nijenhuis vertical tensor field associated to $\phi$.

Let us put 
\begin{equation}
\label{IIIs1}\xi_{2}=\frac{1}{\sqrt{F^2+K^2}}\left(y^i\frac{\partial}{\partial y^i}+p_i\frac{\partial}{\partial p_i}\right)\,,\,\xi_1=\phi(\xi_2)=\frac{1}{\sqrt{F^2+K^2}}\left(p^i\frac{\partial}{\partial y^i}-y_i\frac{\partial}{\partial p_i}\right),
\end{equation}
where as before $y_i=g_{ij}y^j$ and $p^i=g^{ij}p_j$.

Also, we consider the corresponding dual vertical $1$-forms of $\xi_1$ and $\xi_2$, respectively, which are locally given by
\begin{equation}
\label{IIIs2}
\omega^1=\frac{1}{\sqrt{F^2+K^2}}(p_i\theta^i-y^ik_i)\,,\,\omega^2=\frac{1}{\sqrt{F^2+K^2}}(p^ik_i+y_i\theta^i),
\end{equation}
where $\theta^i(\partial/\partial y^j)=\delta^i_j$, $\theta^i(\partial/\partial p_j)=0$, $k_i(\partial/\partial y^j)=0$ and $k_i(\partial/\partial p_j)=\delta^j_i$.

By direct calculations, we have

\begin{lemma}
\label{ls1}
The following assertions hold:
\begin{itemize}
\item[(i)] $\phi(\xi_1)=-\xi_2$, $\phi(\xi_2)=\xi_1$;
\item[(ii)] $\omega^1\circ\phi=\omega^2$, $\omega^2\circ\phi=-\omega^1$;
\item[(iii)] $\omega^a(X)=G(X,\xi_a)$, $a=1,2$.
\end{itemize}
\end{lemma}

Now, we define a tensor field $f$ of type $(1,1)$ on $V$ by
\begin{equation}
f(X)=\phi(X)-\omega^{2}(X)\xi_{1}+\omega^{1}(X)\xi_{2},\,\forall\,X\in\Gamma(V).
\label{IIIx4}
\end{equation}
\begin{theorem}
\label{tIII1}
The triplet $(f,(\xi_{a}),(\omega^{a})), a=1,2$ provides a framed $f(3,1)$-structure on $V$, namely
\begin{enumerate}
\item[(i)]  $\omega^{a}(\xi_{b})=\delta^{a}_{b}$ , $f(\xi_{a})=0$ , $\omega^{a}\circ f=0$;
\item[(ii)] $f^{2}(X)=-X+\omega^{1}(X)\xi_{1}+\omega^{2}(X)\xi_{2}$, for any  $X\in\Gamma (V)$;
\item[(iii)] $f$ is of rank $2n-2$ and $f^{3}+f=0$.
\end{enumerate}
\end{theorem}
\begin{proof}
Using \eqref{IIIx4} and Lemma \ref{ls1} (i) and (ii), by direct calculations we get (i) and (ii). Applying $f$ to the equality (ii) and taking into account the equality (i) one obtains $f^{3}+f=0$. Now, from the second equations in (i), we see that ${\rm span}\{\xi_{1},\xi_{2}\}\subset \ker f$. We prove now that $\ker f\subset {\rm span}\{\xi_1,\xi_2\}$. Indeed, let be $X\in\ker f$ written locally in the form $X=X^i\frac{\partial}{\partial y^i}+Y_i\frac{\partial}{\partial p_i}$. By a direct calculation, the condition $f(X)=0$ gives
\begin{eqnarray*}
X&=&\frac{p_iX^i-y^iY_i}{\sqrt{F^2+K^2}}\xi_1+\frac{y_iX^i+p^iY_i}{\sqrt{F^2+K^2}}\xi_2\in{\rm span}\{\xi_1,\xi_2\}
\end{eqnarray*}
and rank$f=2n-2$.
\end{proof}
\begin{theorem}
\label{tIII2}
The Riemannian metric $G$ verifies
\begin{equation}
G(f(X),f(Y))=G(X,Y)-\omega^{1}(X)\omega^{1}(Y)-\omega^{2}(X)\omega^{2}(Y)
\label{IIIx5}
\end{equation}
for any $X,Y\in\Gamma(V)$.
\end{theorem}
\begin{proof}
Since $G(\xi_{1},\xi_{2})=0$ and $G(\xi_{1},\xi_{1})=G(\xi_{2},\xi_{2})=1$, by using \eqref{IIIx4} and Lemma \ref{ls1} (ii) and (iii) we get \eqref{IIIx5}.
\end{proof}
\begin{remark}
The above theorem follows in a different way if we use the local expression of the vertical tensor field $f$ in the local vertical frame $\{\frac{\partial}{\partial y^i}, \frac{\partial}{\partial p_i}\}$. Indeed, from \eqref{IIIx4} we have
\begin{equation}
f\left(\frac{\partial}{\partial y^i}\right)=\frac{p_iy^j-y_ip^j}{F^2+K^2}\frac{\partial}{\partial y^j}-\left(g_{ij}-\frac{y_iy_j+p_ip_j}{F^2+K^2}\right)\frac{\partial}{\partial p_j},
\label{IIIx6}
\end{equation}
\begin{equation}
f\left(\frac{\partial}{\partial p_i}\right)=\left(g^{ij}-\frac{p^ip^j+y^iy^j}{F^2+K^2}\right)\frac{\partial}{\partial y^j}+\frac{p^iy_j-y^ip_j}{F^2+K^2}\frac{\partial}{\partial p_j}
\label{IIIx6-1}
\end{equation}
and using \eqref{IIIx6} and \eqref{IIIx6-1} one finds
\begin{eqnarray}
G\left(f\left(\frac{\partial}{\partial y^i}\right),f\left(\frac{\partial}{\partial y^j}\right)\right)=g_{ij}-\frac{y_iy_j+p_ip_j}{F^2+K^2},
\nonumber\\
G\left(f\left(\frac{\partial}{\partial y^i}\right),f\left(\frac{\partial}{\partial p_j}\right)\right)=\frac{p_iy^j-y_ip^j}{F^2+K^2}, \,\,\,\,\,\,\,\,\,\,\,\,\,\,\label{IIIx7}\\
G\left(f\left(\frac{\partial}{\partial p_i}\right),f\left(\frac{\partial}{\partial p_j}\right)\right)=g^{ij}-\frac{y^iy^j+p^ip^j}{F^2+K^2}.\nonumber
\end{eqnarray}
Now, from \eqref{IIIx7} easily follows \eqref{IIIx5}.
\end{remark}

Theorem \ref{tIII2} says that $(f,G)$ is a Riemannian framed $f(3,1)$-structure on $V$.

Let us put $\Phi(X,Y)=G(f(X),Y)$ for any $X,Y\in\Gamma(V)$. We have that $\Phi$ is bilinear since $G$ is so, and using Lemma \ref{ls1} (iii) and Theorems \ref{tIII1} and \ref{tIII2}, by direct calculations we have $\Phi(Y,X)=-\Phi(X,Y)$ which says that $\Phi$ is a $2$-form on $V$.

The Theorem shows that the annihilator of $\Phi$ is ${\rm span}\{\xi_1,\xi_2\}$. Also, a direct calculation gives $[\xi_1,\xi_2]=\frac{1}{\sqrt{F^2+K^2}}\xi_1$ which says that the distribution $\{\xi_1,\xi_2\}$ is integrable even if $\Phi$ is not $d_V$-closed, where $d_V$ is the (leafwise) vertical differential on $\mathcal{T}M$. We notice that the annihilator of a $d_V$-closed vertical $2$-form is always integrable. 

A direct calculus in local coordinates, using \eqref{IIIx6} and \eqref{IIIx6-1}, leads to
\begin{equation}
\label{IIIx8}
\Phi\left(\frac{\partial}{\partial y^i},\frac{\partial}{\partial y^j}\right)=\frac{p_iy_j-y_ip_j}{F^2+K^2}\,,\,\Phi\left(\frac{\partial}{\partial y^i},\frac{\partial}{\partial p_j}\right)=-\delta^j_i+\frac{y_iy^j+p_ip^j}{F^2+K^2}\,,\,\Phi\left(\frac{\partial}{\partial p_i},\frac{\partial}{\partial p_j}\right)=\frac{p^iy^j-y^ip^j}{F^2+K^2}.
\end{equation}
On the other hand, we have
\begin{displaymath}
d_V\omega^1\left(\frac{\partial}{\partial y^i},\frac{\partial}{\partial y^j}\right)=\frac{p_iy_j-y_ip_j}{2(F^2+K^2)\sqrt{F^2+K^2}}\,,\,d_V\omega^1\left(\frac{\partial}{\partial p_i},\frac{\partial}{\partial p_j}\right)=\frac{p^iy^j-y^ip^j}{2(F^2+K^2)\sqrt{F^2+K^2}},
\end{displaymath}
\begin{equation}
\label{IIIx9}
d_V\omega^1\left(\frac{\partial}{\partial y^i},\frac{\partial}{\partial p_j}\right)=\frac{1}{2\sqrt{F^2+K^2}}\left(-2\delta^j_i+\frac{y_iy^j+p_ip^j}{F^2+K^2}\right).
\end{equation}
Now, comparing $\Phi$ with $d_V\omega^1$, it follows
\begin{equation}
\label{IIIx10}
\Phi=2\sqrt{F^2+K^2}d_V\omega^1+\varphi,
\end{equation}
where $\varphi=\delta_i^j\theta^i\wedge k_j$. We have that $\frac{\Phi}{\sqrt{F^2+K^2}}$ is $d_V$-closed if and only if $\frac{\varphi}{\sqrt{F^2+K^2}}$ is $d_V$-closed, and it defines an almost presymplectic structure on the vertical Lie algebroid $V$.

\subsection{An almost contact structure on the vertical Liouville distribution}
Let us begin by considering a vertical Liouville distribution on $\mathcal{T}M$ as the complementary orthogonal distribution in $V$ to the line distribution spanned by the  unitary Liouville vector field $\xi_2=\frac{1}{\sqrt{F^2+K^2}}\mathcal{E}$. In \cite{Ida-Po} this distribution is considered in a more general case when the manifold $M$ is endowed with a Finsler structure and for this reason certain proofs are omitted here.  

Let us denote by $\left\{\xi_2\right\}$ the line vector bundle over $\mathcal{T}M$ spanned by $\xi_2$ and we define the \textit{vertical Liouville distribution} as the complementary orthogonal distribution $V_{\xi_2}$ to $\left\{\xi_2\right\}$ in $V$ with respect to $G$, that is $V=V_{\xi_2}\oplus\{\xi_2\}$. Thus, $V_{\xi_2}$ is defined by $\omega^2$, that is
\begin{equation}
\Gamma\left(V_{\xi_2}\right)=\{X\in\Gamma(V)\,:\,\omega^2(X)=0\}.
\label{a1}
\end{equation}
We get that every vertical vector field $X=X_1^{i}(x,y,p)\frac{\partial}{\partial y^{i}}+X^2_i(x,y,p)\frac{\partial}{\partial p_i}$ can be expressed as:
\begin{equation}
X=PX+\omega^2(X)\xi_2,
\label{a2}
\end{equation}
where $P$ is the projection morphism of $V$ on $V_{\xi_2}$. Also, by direct calculus, we get
\begin{equation}
G(X,PY)=G(PX,PY)=G(X,Y)-\omega^2(X)\omega^2(Y),\,\,\forall\,X,Y\in\Gamma(V).
\label{a3}
\end{equation}
With respect  to the basis $\left\{\theta^{j}\otimes\frac{\partial}{\partial y^i}, \theta^{j}\otimes\frac{\partial}{\partial p_i}, k_{j}\otimes\frac{\partial}{\partial y^i}, k_{j}\otimes\frac{\partial}{\partial p_i}\right\}$ of $\Gamma(V\otimes V^*)$ the vertical tensor field  $P$ is locally given by
\begin{equation}
P=\stackrel{1}{P^{i}_j}\theta^{j}\otimes\frac{\partial}{\partial y^{i}}+\stackrel{2}{P^{j}_i}k_j\otimes\frac{\partial}{\partial p_{i}}+\stackrel{3}{P_{ij}}\theta^{j}\otimes\frac{\partial}{\partial p_{i}}+\stackrel{4}{P^{ij}}k_j\otimes\frac{\partial}{\partial y^{i}},
\label{a4}
\end{equation}
where the local components are expressed by
\begin{equation}
\label{a5}
\stackrel{1}{P^{i}_j}=\delta^{i}_j-\frac{y_jy^{i}}{F^2+K^2}\,,\,\stackrel{2}{P^{i}_j}=\delta^{i}_j-\frac{p^ip_j}{F^2+K^2}\,,\,\stackrel{3}{P_{ij}}=-\frac{y_jp_i}{F^2+K^2}\,,\,\stackrel{4}{P^{ij}}=-\frac{p^jy^{i}}{F^2+K^2}.
\end{equation}

\begin{theorem}
\label{integrabil}
The vertical Liouville distribution $V_{\mathcal{E}}$ is integrable and it defines a Lie algebroid structure on $\mathcal{T}M$, called vertical Liouville Lie algebroid over the big-tangent manifold $\mathcal{T}M$.
\end{theorem}
\begin{proof}
Follows using an argument similar to that used in \cite{B-F1, I-M}. It can be found in \cite{Ida-Po} for a more general case  when the manifold $M$ is endowed with a Finsler structure.
\end{proof}

Now, let us restrict to $V_{\xi_2}$ all the geometrical structures introduced in Section 2 for $V$, and we indicate this by overlines. Hence, we have
\begin{enumerate}
\item[$\bullet$] $\overline{\xi_1}=\xi_1$ since $\xi_1$ lies in $V_{\xi_2}$;
\item[$\bullet$] $\overline{\omega^2}=0$ since $\omega^2(X)=G(X,\xi_2)=0$ for every vertical vector field $X\in V_{\xi_2}$;
\item[$\bullet$] $\overline{G}=G|_{V_{\xi_2}}$;
\item[$\bullet$] $\overline{f}(X)=\overline{\phi}(X)+\overline{\omega^1}(X)\otimes\xi_2$ is an endomorphism of $V_{\xi_2}$ since 
\begin{displaymath}G\left(\overline{f}(X),\xi_2\right)=G(\overline{\phi}(X),\xi_2)+\overline{\omega^1}(X)G(\xi_2,\xi_2)=\omega^2(\overline{\phi}(X))+\overline{\omega^1}(X)=0.
\end{displaymath}
\end{enumerate}

We denote now $\overline{\xi}=\overline{\xi_1}$ and $\overline{\eta}=\overline{\omega^1}$. By Theorem \ref{tIII1} we obtain
\begin{theorem}
\label{tc}
The triple $(\overline{f},\overline{\xi},\overline{\eta})$ provides an almost contact structure on $V_{\xi_2}$, that is
\begin{enumerate}
\item[(i)] $\overline{f}^3+\overline{f}=0$, ${\rm rank}\overline{f}=2n-2=(2n-1)-1$;
\item[(ii)] $\overline{\eta}(\overline{\xi})=1$, $\overline{f}(\overline{\xi})=0$, $\overline{\eta}\circ\overline{f}=0$;
\item[(iii)] $\overline{f}^2(X)=-X+\overline{\eta}(X)\overline{\xi}$, for $X\in V_{\xi_2}$.
\end{enumerate}
\end{theorem}
Also, by Theorem \ref{tIII2} we obtain
\begin{theorem}
\label{thacs}
The Riemannian metric $\overline{G}$ verifies
\begin{equation}
\label{V7}
\overline{G}(\overline{f}(X),\overline{f}(Y))=\overline{G}(X,Y)-\overline{\eta}(X)\overline{\eta}(Y),
\end{equation}
for every vertical vector fields $X,Y\in V_{\xi_2}$.
\end{theorem}
Concluding, as $V_{\xi_2}$ is an integrable distribution, the ensemble $(\overline{f},\overline{\xi},\overline{\eta},\overline{G})$ is an almost contact Riemannian structure on the Lie algebroid $V_{\xi_2}$.

Let us consider now $\overline{\Phi}(X,Y)=\overline{G}(\overline{f}(X),Y)$, for $X,Y\in\Gamma(V_{\xi_2})$, the vertical $2$-form usually associated to the almost contact Riemannian structure from Theorem \ref{thacs}.

The vertical Liouville distribution $V_{\xi_2}$ is spanned by $\left\{P(\frac{\partial}{\partial y^{i}}),P(\frac{\partial}{\partial p_i})\right\}$, where by using \eqref{a4}, we have
\begin{equation}
\label{V8}
P(\frac{\partial}{\partial y^{i}})=\stackrel{1}{P^l_i}\frac{\partial}{\partial y^l}+\stackrel{3}{P_{li}}\frac{\partial}{\partial p_l}\,,\,P(\frac{\partial}{\partial p_{j}})=\stackrel{2}{P^j_k}\frac{\partial}{\partial p_k}+\stackrel{4}{P^{kj}}\frac{\partial}{\partial y^k}.
\end{equation}
Now, if we denote by $\overline{d}_V=d_V|_{V_{\xi_2}}$ then, by direct calculations in the basis $\left\{P(\frac{\partial}{\partial y^{i}}),P(\frac{\partial}{\partial p_i})\right\}$, we get
\begin{equation}
\label{V13}
\overline{d}_V\overline{\eta}=\frac{\overline{\Phi}}{\sqrt{F^2+K^2}}.
\end{equation}
\begin{remark}
The relation \eqref{V13} can be obtained directly from \eqref{IIIx10}, since a straightforward computation shows that $\overline{\varphi}=\varphi|_{V_{\xi_2}}=0$.
\end{remark}
Finally, $\overline{\eta}\wedge\left(\overline{d}_V\overline{\eta}\right)^{n-1}=\overline{\eta}\wedge\left(\frac{\overline{\Phi}}{\sqrt{F^2+K^2}}\right)^{n-1}\neq0$, which says that $\left(\overline{\eta},\frac{\overline{\Phi}}{\sqrt{F^2+K^2}}\right)$ is a contact structure on the vertical Liouville Lie algebroid $V_{\xi_2}$.

\begin{remark}
If $A$ is a Lie algebroid then it is well known that $A\oplus A^*$ has a natural structure of a Courant algebroid, and contact structures on Courant Lie algebroids are recently considered in \cite{Gra}. On the other hand if we consider a Riemannian vector bundle $(A,g_A)$ and $\mathcal{A}$ is the total space of the vector bundle $A\oplus A^*\rightarrow M$, then in a similar manner with our study we can construct an almost contact Riemannian structure on an integrable vertical Liouville distribution over $\mathcal{A}$. However, the most techniques used in the study of the geometry of the total space of a vector bundle $A$ (or its dual $A^*$) have some analogies (for the case of Lie algebroids) when investigate the geometry of the prolongations $\mathcal{T}^AA$ and $\mathcal{T}^AA^*$, respectively, and then, we can formulate the following problem: $\mathcal{T}^AA$ is a Lie algebroid over $A$ and $\mathcal{T}^AA^*$ is a Lie algebroid over $A^*$, and thus, in place of direct sum we can consider the direct product $\mathcal{T}^AA\times \mathcal{T}^AA^*\rightarrow A\times A^*$ (viewed as direct sum of $pr_1^{-1}(\mathcal{T}^AA)\rightarrow A\times A^*$ and of $pr_2^{-1}(\mathcal{T}^AA^*)\rightarrow A\times A^*$, where $pr_1:A\times A^*\rightarrow A$ and $pr_2:A\times A^*\rightarrow A^*$). In this way, we can consider a vertical subbundle of $\mathcal{T}^AA\times \mathcal{T}^AA^*$ as $V\left(\mathcal{T}^AA\times \mathcal{T}^AA^*\right)=pr_1^{-1}(V^AA)\oplus pr_2^{-1}(V^AA^*)$, where $V^AA$ is the vertical subbundle of $\mathcal{T}^AA$ and $V^AA^*$ is the vertical subbundle of $\mathcal{T}^AA^*$. Moreover, we can consider a Liouville (Euler) section of $V\left(\mathcal{T}^AA\times \mathcal{T}^AA^*\right)$ as direct sum of the canonical Liouville sections of $pr_1^{-1}(V^AA)$ and $pr_2^{-1}(V^AA^*)$ and a Liouville type subbundle of $V\left(\mathcal{T}^AA\times \mathcal{T}^AA^*\right)$ defined as the orthogonal subbundle of $V\left(\mathcal{T}^AA\times \mathcal{T}^AA^*\right)$ to the line bundle generated by Liouville section. Then, another problem to solve is the construction of an almost contact Riemannian structure on the vertical Liouville subbundle using the above procedure. 
\end{remark}

\section*{Acknowledgement}
The authors cordially thanks to Referees for useful comments and suggestions about the initial submission which improve substantially the presentation and the contents of this paper.  Also, many thanks go to Professor Gheorghe Piti\c{s} for fruitful conversations concerning this topic.

\noindent 
Cristian Ida \\
Department of Mathematics and Computer Science\\
Transilvania University of Bra\c{s}ov\\
Address: Bra\c{s}ov 500091, Str. Iuliu Maniu 50, Rom\^{a}nia\\
email:\textit{cristian.ida@unitbv.ro}
\medskip

\noindent 
Paul Popescu\\
Department of Applied Mathematics, University of Craiova\\
Address: Craiova, 200585,  Str. Al. Cuza, No. 13,  Rom\^{a}nia\\
 email:\textit{paul$_{-}$p$_{-}$popescu@yahoo.com}

\end{document}